\documentclass[12pt]{amsart}
\usepackage{amsmath}
\usepackage{amssymb}
\newtheorem{theorem}{Theorem}
\theoremstyle{plain}

\newtheorem{corollary}{Corollary}

\newtheorem{example}{Example}

\newtheorem{lemma}{Lemma}

\newtheorem{proposition}{Proposition}

\numberwithin{equation}{section}

\usepackage{hyperref}

\begin{document}

\title{Simplest cubic fields}
\author{Q. Mushtaq}
\address{Department of Mathematics, Quaid-i-Azam University, Islamabad,
Pakistan}
\email{qmushtaq@apllo.net.pk}
\author{S. Iqbal}

\begin{abstract}
Let $Q(\alpha)$ be the simplest cubic field, it is known that $Q(\alpha)$ can be generated by adjoining a root of the irreducible equation $x^{3}-kx^{2}+(k-3)x+1=0$, where $k$ belongs to $Q$. In this paper we have established a relationship between $\alpha$, $\alpha'$  and $k,k'$ where $\alpha$ is a root of the equation $x^{3}-kx^{2}+(k-3)x+1=0$ and $\alpha'$ is a root of the same equation with $k$ replaced by $k'$ and $Q(\alpha)=Q(\alpha')$.
\end{abstract}

\maketitle
\section{Introduction}
Every simplest cubic field can be generated by adjoining a root of the equation \cite{sd}
\begin{equation}
x^{3}-kx^{2}+(k-3)x+1=0,  \tag{1}
\end{equation}

for some suitable \(k\), where $k$ is a parameter belonging to set of rational numbers. Conversely,
the roots of $(1)$, for some rational number $k$, are either rational or generate a simplest cubic field.

It can be observed that two different values of parameters  can generate the same simplest cubic field. In this paper we will classify the parameters according to the corresponding  simplest cubic fields. We also found the relation between the roots of equation (1) for different values of parameters belonging to the same class.

We assert that these relations could be a step forward to understand about integral points on elliptic curves related to simples cubic fields \cite{ds}

Each element $\omega $ of $Q(\alpha )$ has
the linear fractional representation \(\omega =\frac{a\rho+b}{c\rho +d}\),
where the determinant of $\omega $, denoted by $\det (\omega )$, is defined
as $ad-bc.$ An element $\omega $ of $Q(\alpha )$ is rational if and only if $\det(\omega )=0$.

\section{Relationship between $\protect\alpha $ and $\acute{\protect\alpha}$
\ when $Q(\protect\alpha )=Q(\protect\alpha ^{\prime })$}

First we split the simplest cubic field into classes. Each element$ \omega $
of $Q(\alpha)$ is related to a linear fractional transformation $f \in PSL(2,Q)$ in the
way that all its conjugates can be obtained by applying $f^n$, for some
integer $n$, over $\omega $. We shall call this correspondence as ``linear
fractional transformation related to an element of $Q(\alpha)$''.
This correspondence defines an equivalence relation on $Q(\alpha )$ and each class corresponds to a unique element of \(PSL(2,Q)\). 

We call the equivalence class related to a linear fractional transformation $ f$ as $f$-class. The following corollary is straightforward to see.

\begin{corollary}
If $\omega $ $\in $ $Q(\rho )$, then $\omega $ and its conjugates belong to the same $f$-class.
\end{corollary}

The elements belonging to $y$-class, where \(y\) is defined as \(z \longmapsto \frac{z-1}{z}\),  are instrumental in establishing a relationship between $\alpha, \alpha'$ \ and $k,k'$ where $\alpha $ is a root of Equation 1 and $\alpha'$ is a root of same equation with $k$ replaced by $k'$ and $Q(\alpha )=Q(\alpha')$. The following example shows an element belonging to $y$-class.

\begin{example}
The roots of the equation $x^{3}-3x^{2}+1=0$ are\ given by 
\begin{equation*}
\rho \text{, }\frac{\rho -1}{\rho }\text{, }\frac{-1}{\rho -1}\text{.}
\end{equation*}%
where $\rho =2\cos (\pi /9)+1$. The algebraic number $\rho $ belongs to $y$%
-class
\end{example}

Note that in $y$-class there are more than one elements, for instance, for $%
k=3$ the roots of Equation 1, are $\rho =2\cos (\pi /9)+1$, $\rho 
{\acute{}}
=-\cos (\pi /9)+1-\sqrt{3}\sin (\pi /9)$ and $\rho ^{\prime \prime }=-\cos
(\pi /9)+1+\sqrt{3}\sin (\pi /9)$ and the following two numbers belong to $y$-class
\begin{equation*}
\frac{5\rho -2}{2\rho +3}\text{and }\frac{3\rho -1}{\rho +2}
\end{equation*}%
and the following equations 
\begin{equation*}
x^{3}-\left( -\frac{51}{73}\right) x^{2}+(-\frac{51}{73}-3)x+1=0\text{, }
\end{equation*}%
\begin{equation*}
x^{3}-\left( \frac{3}{19}\right) x^{2}+(\frac{3}{19}-3)x+1=0
\end{equation*}

are respectively satisfied by them.

\begin{lemma}
An element $\omega $ of $Q( \alpha )$ belong to $y$-class if and only if it is a root of the irreducible equation $x^{3}-kx^{2}+(k-3)x+1=0$, where $k$ is a rational number.
\end{lemma}

\begin{proof}
Let $\omega$ be an element of $y$-class so the other conjugates of $\omega$
are $\frac{\omega -1}{\omega }$, $\frac{-1}{\omega -1}$ and the equation satisfied by them is: 
\begin{equation*}
(x-\omega)(x-\frac{\omega -1}{\omega })(x+\frac{1}{\omega -1})=0
\end{equation*}
which can be simplified as 
\begin{equation*}
x^{3}-kx^{2}+(k-3)x+1=0
\end{equation*}
where $k=\omega +\frac{\omega -1}{\omega }-\frac{1}{\omega -1}$. Here $k$ must be rational because it is evolved from the equation satisfied by $\omega$,
which is algebraic over $Q$ of degree three.

Conversely, let $\omega$ satisfy an irreducible equation of the form $x^{3}-kx^{2}+(k-3)x+1=0$, where $k$ is a rational number. Then the other two roots of the equation are $\frac{\omega -1}{\omega }$ and $\frac{-1}{\omega -1}$.
Moreover, $(\omega )y=\frac{(\omega -1)}{\omega }$, $(\frac{\omega -1}{\omega})y=\frac{-1}{(\omega -1)}$ and $(\frac{-1}{\omega -1})y=\omega$ imply that $\omega$ and its conjugates lie in the same triangle.
\end{proof}

Thus, corresponding to each ``triplet of conjugates" of $y$-class there is a rational number $k$. We call such correspondence as the element of $y$-class with parameter $k$.

We use the above classification of elements of $y$-class to prove the
following.

\begin{proposition}
Elements of $Q(\alpha)$ belonging to $y$-class are of the form $\frac{(c+d)\alpha -c}{c\alpha+d}$, where $c$ and $d$ are non-zero integers.
\end{proposition}

\begin{proof}
Let $\omega =\frac{a\alpha +b}{c\alpha +d}$ be any primitive element in $Q(\alpha)$, that is, $ad-bc\neq 0$. Then the conjugates of $\omega $ are given by 
\begin{equation*}
\omega'=\frac{(a+b)\alpha -a}{(c+d)\alpha -c}\text{,}
\end{equation*}
and 

\begin{equation*}
\omega''=\frac{b\alpha -(a+b)}{d\alpha -(c+d)}.
\end{equation*}

Our next aim is to find the values of $a,b,c$ and $d$ such that $\omega$ belongs to \(y\)-class, that is, when 
\begin{eqnarray}
(\omega )y &=&\omega'\\
(\omega')y &=&\omega'' \notag \\
(\omega'')y &=&\omega  \notag.
\end{eqnarray}
Now from the first of these equations, we get 
\begin{equation*}
\frac{(a-c)\alpha +(b-d)}{a\alpha +b}=\frac{\lambda (a+b)\alpha -\lambda a}{\lambda (c+d)\alpha -\lambda c}\text{,}
\end{equation*}
which further yields the following four equations 
\begin{equation*}
\begin{array}{cc}

a(1-\lambda)-c-\lambda b=0 \text{, } & \lambda a+b-d=0, \\ 
a-\lambda (c+d)=0\text{ and } & b+\lambda c=0.
\end{array}
\end{equation*}
But the $det(\frac{\lambda (a+b)\alpha -\lambda a}{\lambda (c+d)\alpha -\lambda c})=det(\omega )$ which is possible if $\lambda = \pm 1$. For $\lambda =-1$,
the above system of equations gives $a=b=c=d=0$, contradicting the
primitivity of $\omega $. The solution of the system for $\lambda =1$ is $a=c+d$, $b=-c$, $c=c$, $d=d$. This solution satisfies the other two
Equations of (2.1). So the elements of \(y\)-class are of the form $\frac{(c+d)\alpha -c}{c\alpha +d}$.
\end{proof}

Now we are in a position to state the theorem which establishes a
relationship between $\alpha $ and $\alpha ^{\prime }$ whenever $Q(\alpha
)=Q(\alpha ^{\prime })$.

\begin{theorem}
If $\alpha $ is a root of the irreducible equation $x^3-kx^2+(k-3)x+1=0$
and $\alpha'$is a root of the same irreducible equation except $k$ is replaced by $k'$ then $Q(\alpha )=Q(\alpha')$ if and only if $\alpha'$ is of the form $\frac{(c+d)\alpha -c}{c\alpha +d}$ or $\alpha$ is of the form $\frac{(c+d)\alpha'-c}{c\alpha'+d}$ for some integer values of $c$ and $d$, not both zero.
\end{theorem}
\subsection{Relationship between $k$ and $k'$ \ when $Q(\alpha )=Q(\alpha')$}
Hence we will answer the next question about  simplest cubic field. First we will find the equation satisfied by any element $\omega =%
\frac{a\rho +b}{c\rho +d}$ of the simplest cubic field $Q(\rho )$ where $\rho$
is the root of the equation $x^{3}-kx^{2}+(k-3)x+1=0$. The representation of $\omega $ as a linear combination of $\{1,\rho ,\rho ^{2}\}$ obtained by the method of indeterminate coefficients is given by 
\begin{equation*}
\omega =\frac{a_{1}^{2}\rho +b_{1}\rho +c_{1}}{d_{1}}
\end{equation*}

where 

\begin{eqnarray*}
a_{1} &=&(-c(bc-da)) \\
b_{1} &=&((ck+d)(bc-da)) \\
c_{1} &=&(ac^{2}-bc^{2}k+3bc^{2}-bdck-d^{2}b) \\
d_{1} &=&(-c^{2}kd+3c^{2}d-d^{2}ck-d^{3}+c^{3})
\end{eqnarray*}

are rational integers. Then equation satisfied by $\omega $ is: 

\begin{equation*}
z^{3}+z^{2}\left( \frac{2bdck+3d^{2}b+bc^{2}k-3ac^{2}-3bc^{2}+kd^{2}a+2ckda-6cda}{-c^{2}kd+3c^{2}d-d^{2}ck-d^{3}+c^{3}}\right)
\end{equation*}

\begin{equation*}
-z\left( \frac{-3ca^{2}+cb^{2}k-6cba+2cbka+2bdka+a^{2}kd+3db^{2}-3a^{2}d}{-c^{2}kd+3c^{2}d-d^{2}ck-d^{3}+c^{3}}\right)
\end{equation*}
\begin{equation*}
+\left( \frac{-3ba^{2}+b^{3}+b^{2}ka+bka^{2}-a^{3}}{-c^{2}kd+3c^{2}d-d^{2}ck-d^{3}+c^{3}}\right) =0\text{,}
\end{equation*}

Equation satisfied by elements of $y$-class can be obtained by putting $a=c+d $, $b=-c$ in above equation 

\begin{equation*}
z^{3}-z^{2}t+z(t-3)+1=0\text{,}
\end{equation*}

where $t$ is given by 

\begin{equation*}
t=\frac{9d^{2}c+c^{3}k+9c^{2}d-3d^{2}ck-kd^{3}}{-c^{2}kd+3c^{2}d-d^{2}ck-d^{3}+c^{3}}\text{.}
\end{equation*}

The following three elements

\begin{equation*}
w=(\frac{(c+d)\alpha -c}{c\alpha +d}),w
{\acute{}}
=(\frac{d\alpha -(c+d)}{(c+d)\alpha -c}),w^{\prime \prime }=(\frac{-c\alpha
-d}{d\alpha -(c+d)})
\end{equation*}

are the roots of the Equation 3.2. Hence if $\alpha $ is a root of $x^{3}-kx^{2}+(k-3)x+1=0$ and $\alpha 
{\acute{}}
$ is a root of the same equation, except $k$ is replaced by $k$, such that $Q(\alpha )=Q(\alpha ^{\prime })$ then the relationship between 
$k$ and $k
{\acute{}}
$ is given by 
\begin{equation*}
k
{\acute{}}
=\frac{9d^{2}c+c^{3}k+9c^{2}d-3d^{2}ck-kd^{3}}{-c^{2}kd+3c^{2}d-d^{2}ck-d^{3}+c^{3}}\text{.}
\end{equation*}

From now onwards, we shall write the above relation between $k$ and $k'$ as $k'=T(c,d,k)$.

Hence we can state:

\begin{theorem}
If $\alpha $ is a root of the equation $x^{3}-kx^{2}+(k-3)x+1=0$ and $\alpha' $
 is a root of the same equation, except $k$ is replaced by $k'$, then $Q(\alpha )=Q(\alpha')$ if and only if \(k'=T(c,d,k)\)
for integers $c$ and $d$ not both zero.

\end{theorem}

We define a relation \(R\) on set of rational numbers as  $aRb$ if there exist integers \(c\) and \(d\) such that $k'=T(c,d,k)$.
It can be easily seen that the relation \(R\) is equivalence relation. Hence,

\begin{theorem}
 There is one-to-one correspondence between the  equivalence classes of \(R\) and distinct simplest cubic
fields and vice versa, with the exception of one equivalence class $\{k:$ $k=\frac{(p^{3}-3pq^{2}+q^{3})}{(qp\left( -q+p\right) )}$, for some integers $p$ and $q$ such that $0\neq p\neq q\neq 0\}.$
\end{theorem}

\end{document}